\documentclass[a4paper,11pt]{amsart}
\usepackage{lmodern}
\usepackage[T1]{fontenc}
\usepackage{amsmath, amsthm, amssymb, mathrsfs, longtable, multirow}
\usepackage[inline]{enumitem}
\usepackage{graphicx}

\title%
[]
{
An extension of Pizzetti's formula associated with the Dunkl operators
}
\author{Nobukazu Shimeno}
\email{shimeno@kwansei.ac.jp}
\address{School of Science and Technology, Kwansei Gakuin University, 
2-1 Gakuen, Sanda, Hyogo 669-1337, Japan}
\author{Naoya Tani}
\email{}
\address{ Graduate School of Science and Technology, Kwansei Gakuin University, 
2-1 Gakuen, Sanda, Hyogo 669-1337, Japan}

\dedicatory{To the memory of  Professor~Takaaki~Nomura} 
\date{}

\newcommand{\D}{\mathscr{D}}
\renewcommand{\P}{\mathscr{P}}
\renewcommand{\H}{\mathscr{H}}
\newcommand{\kk}{\kappa}

\numberwithin{equation}{section}
\theoremstyle{plain}
 \newtheorem{thm}{Theorem}[section]
 \newtheorem{cor}[thm]{Corollary}
 
 \newtheorem{prop}[thm]{Proposition}
\theoremstyle{definition}

\theoremstyle{remark}
 \newtheorem{rem}[thm]{Remark}

\calclayout

\allowdisplaybreaks[3]

\begin{document}
\maketitle

\begin{abstract}
We give an extension of Pizzetti's formula associated with the Dunkl operators. 
It gives an explicit formula for the Dunkl
 inner product of an arbitrary function and a homogeneous Dunkl  harmonic polynomial on the unit sphere. 
\end{abstract}

\section{Introduction}

Dunkl analysis that was initiated by C.~Dunkl is a study of the function theory for the Dunkl operators and 
the Dunkl Laplacian. It is a deformation of calculus of several variables for partial derivatives and 
the Euclidean Laplacian, and analogues of classical results have been developed. 

In this paper we study analogues in Dunkl analysis 
of classical Pizzetti's formula for spherical mean \cite{pizzetti, poritsky, 
lysik} and its extension given by Bezubik, D\c{a}browska, and Strasburger 
\cite[Corollary~2.1]{BDS} and Estrada~\cite{est2}.
The main result is an extended Pizzetti's formula associated with the Dunkl operators 
(Theorem~\ref{thm:expiz}):
\begin{equation*}
\dfrac{1}{\omega_{\kappa, d}} \int_{S^{d-1}} q(y) f(ry) h_{\kappa}^{2}(y) d\sigma(y)
 =\sum_{n=0}^N \frac{(q(\mathscr{D}) \Delta_{\kappa}^{n} f)(0)}{n!\,(\lambda_{\kk}+1)_{m+n}} 
 \left(\dfrac{r}{2}\right)^{m+2n}+o(r^{m+2N}) \quad (r\to +0)
\end{equation*}
for a homogeneous Dunkl harmonic polynomial $q$ of degree $m$ and 
a smooth function $f$ on $\mathbb{R}^d$. Here $\kk$ is a parameter associated with a finite reflection group. 
For $\kk=0$ the above formula gives the classical result (\cite{BDS, est2}) with $\omega_{0,d}$ is the surface 
area of the unit sphere $S^{d-1}$, $h_0(y)=1$, $\D$ for $\kk=0$ is the gradient, $\Delta_0$ is the Euclidean Laplacian, 
and 
$\lambda_0=d/2-1$. In particular, the case of $\kk=0$ and $q(y)=1$ gives classical Pizzetti's formula. 
See the next section for explanation of the notions undefined here. 

We prove our extended Pizzetti's formula by using the canonical decomposition of a homogeneous polynomial 
with respect to the Dunkl harmonic polynomials. Pizzetti's formula follows as a corollary of 
the extended Pizzetti's formula. Pizzetti's formula associated with the Dunkl Laplacian was established by Mejjaoli,  Trim\`eche~\cite[Theorem~4.17]{MT} 
and Salem, Touahri~\cite[Theorem~2.4]{ST}. Our proof is different from theirs. 
We also deduce the extended Pizzetti's formula from 
Pizzetti's formula and Hobson's formula~\cite{shimeno}. 

As a corollary of our extended Pizzetti's formula, we obtained the Funk-Hecke formula associated 
with the Dunkl operators, which was originally proved by Xu~\cite{xu} by a different method. 

\section{Notation and preliminaries}

In this section we review  the Dunkl operators and Dunkl $h$-harmonics. 
We refer \cite{DaX, DX, R} for details. 

Let $d$ be a positive integer.  
Let $\langle \,,\,\rangle$ be the standard inner product on $\mathbb{R}^d$ and put $||x||=\langle x, x\rangle^{1/2}$ 
for $x\in\mathbb{R}^d$.  Let 
$R\subset \mathbb{R}^d$ be a reduced root system, which is not necessarily crystallographic. 
For $\alpha\in R$, we write 
$r_\alpha$ for the reflection with respect to the hyperplane $\alpha^\perp$. 
Let $G$ denote the finite reflection group generated by $\{r_\alpha\,:\,\alpha\in R\}$. 
We fix a positive system $R_+\subset R$.

Let $\mathscr{P}=\mathscr{P}(\mathbb{R}^d)$ denote the space of polynomials on $\mathbb{R}^d$ with 
real coefficients. 
For a non-negative integer $m$, let $\mathscr{P}_m$ denote the subspace of $\mathscr{P}$ consisting of the 
homogeneous 
polynomials of degree $m$. 

Let $\kk:R\rightarrow \mathbb{R}_{\geq 0},\,\alpha\mapsto \kk_\alpha$ be a $G$-invariant 
function on $R$. We call $\kk$ a (non-negative) multiplicity function. 
We define
\begin{equation}\label{eqn:const}
\lambda_\kk=\frac{d}{2}-1+\sum_{\alpha\in R_+}\kk_\alpha.
\end{equation}

For $\xi\in\mathbb{R}^d\setminus\{0\}$, let 
$\partial_\xi
$ denote the directional derivative corresponding to  $\xi$ and define
the Dunkl operator $\D_\xi$ by
\begin{equation}\label{eqn:dunkl}
\D_\xi f(x)=\partial_\xi f(x)+\sum_{\alpha\in R_+}\kk_\alpha\langle\alpha,\xi\rangle\frac{f(x)-f(r_\alpha x)}
{\langle\alpha,x\rangle}. 
\end{equation}
The Dunkl operators satisfy 
$\D_\xi\D_\eta=\D_\eta\D_\xi$ for all $\xi,\,\eta\in\mathbb{R}^d$. 
Let $\{e_1,\dots,e_d\}$ be the standard orthonormal basis of $\mathbb{R}^d$. 
We write $\partial_j=\partial_{e_j},\,\,\D_j=\D_{e_j}$. The 
Dunkl Laplacian $\Delta_\kk$ is defined by
\begin{equation}
\Delta_\kk=\sum_{j=1}^d \D_{j}^2.
\end{equation}
The Dunkl operators are homogeneous of degree $-1$ and the Dunkl Laplacian $\Delta_\kk$ is 
homogeneous of degree $-2$. 

Put $\D=(\D_1,\dots,\D_d)$. 
For $p,\,q\in\mathscr{P}$ define $\langle p, q\rangle_\kk=(p(\D)q)(0)$. Then $\langle\,\cdot\,,\,\cdot\,\rangle_\kk$ 
gives a non-degenerate symmetric bilinear form on $\mathscr{P}$. If $p\in\P_l,\,q\in\P_m$ and $l\not=m$, then 
$\langle p,q\rangle_\kk=0$. 

Define $\mathscr{H}_{\kappa} := \left\{ p \in \mathscr{P} \,  ; \, \Delta_{\kappa} p = 0  \right\}$ and 
$\mathscr{H}_{\kappa, m} = \mathscr{H}_{\kappa} \cap \mathscr{P}_{m}$. We call an element of $\H_\kk$ 
an $h$-harmonic polynomial or 
Dunkl harmonic polynomial. 

We recall the canonical decomposition of a homogeneous polynomial. Define the shifted factorial 
\begin{align*}
& (a)_0=1, \\
& (a)_n=a(a+1)\cdots (a+n-1), 
\end{align*}
where $a$ is a complex number and $n$ is a positive integer. 
For $P\in\mathscr{P}_n$ define
\[
\text{proj}_{\kk,n} P(x)= \sum_{j=0}^{[n/2]} \dfrac{1}{4^{j}j!\, (-\lambda_{\kk}-n+1 )_{j}} 
||x||^{2j} \Delta_{\kappa}^{j} P(x).
\]
Here $[n/2]$ means the largest integer with $[n/2]\leq n/2$. 

\begin{thm}[Canonical decomposition,  {\cite[Theorem~1.11]{Dunkl3},  \cite[Theorem~7.1.15]{DX}}]\label{thm:cd}
We have an orthogonal direct sum decomposition
\begin{equation*}
\mathscr{P}_{n} = \mathscr{H}_{\kappa, n} \oplus ||x||^{2}\mathscr{H}_{\kappa, n-2} \oplus \cdots 
\oplus ||x||^{2[n/2]}\mathscr{H}_{\kappa, n-2[n/2]} .
\end{equation*}
The decomposition of $p\in\mathscr{P}_n$ is given by 
\[
 p(x)=\sum_{i=0}^{[n/2]}||x||^{2i}p_{n-2i}(x)\quad (p_{n-2i}\in\mathscr{H}_{\kk,n-2i}) 
\]
with
\[
p_{n-2i}(x)=\frac{1}{4^ii!\,(\lambda_{\kk}+1+n-2i)_i}\mathrm{proj}_{\kk,n-2i}\Delta_{\kk}^i p(x) . 
\]
\end{thm}

Let $h_\kk(x)$ denote the weight function defined by
\begin{equation}\label{eqn:wt}
h_\kappa(x)=\prod_{\alpha\in R_+}|\langle\alpha,x\rangle|^{\kk_\alpha}
\end{equation}

Let $S^{d-1}$ denote the unit sphere in $\mathbb{R}^d$ and 
$d\sigma$  the surface measure on $S^{d-1}$. Define
\begin{equation*}
\omega_{\kappa, d} = \int_{S^{d-1}} h_{\kappa}^2(y) d\sigma(y).
\end{equation*}

We recall the orthogonality relation for $h$-spherical harmonics.

\begin{thm}[{\cite[Theorem~3.1.2, Theorem~3.1.9]{DaX}}]\label{thm:orthogonality}
Suppose $p\in \mathscr{H}_{\kk,l}$ and $q\in\mathscr{H}_{\kk,m}$. 
Then  
\[
\dfrac{1}{\omega_{\kappa, d}} \int_{S^{d-1}} p(y)q(y)  h_{\kappa}^2(y) d\sigma(y)
= \dfrac{1}{2^{m}\,(\lambda_{\kk}+1)_{m}} 
\langle p,q\rangle_\kk.
\]
\end{thm}
Note that in Theorem~\ref{thm:orthogonality} $\langle p,q\rangle_\kk=\delta_{lm}q(\D)p$. 
If $l<m$, then by Theorem~\ref{thm:cd} and Theorem~\ref{thm:orthogonality}, 
\[
\dfrac{1}{\omega_{\kappa, d}} \int_{S^{d-1}} p(y)q(y)  h_{\kappa}^2(y) d\sigma(y)
=0\quad (p\in \P_l,\,q\in \mathscr{H}_{\kk,m}).
\]

\section{Extended Pizzetti's formula}

In this section, we prove Pizzetti's formula and its extension associated with the Dunkl 
operators.  We start with the following mean value property for homogeneous polynomials. 

\begin{thm}[\cite{tani}]\label{thm:pizzettid1}
Let $p \in \mathscr{P}_{l}$ and $q \in \mathscr{H}_{\kappa, m}$. 
Assume  $l-m$ is a non-negative even integer and set $l-m = 2n$. Then
\begin{equation*}
\dfrac{1}{\omega_{\kappa, d}} \int_{S^{d-1}} q(y) p(y) h_{\kappa}^{2}(y) d\sigma(y)
 = \dfrac{1}{2^{m+2n}n!\,(\lambda_{\kk}+1)_{m+n}} 
q(\mathscr{D}) \Delta_{\kappa}^{n} p.
\end{equation*}
The integral on the left hand side of the above identity vanishes if $l-m$ is a negative or odd integer. 
\end{thm}

\begin{proof}
By 
Theorem~\ref{thm:cd} and 
Theorem~\ref{thm:orthogonality}, 
\[
\dfrac{1}{\omega_{\kappa, d}} \int_{S^{d-1}} q(y) p(y) h_{\kappa}^{2}(y) d\sigma(y)=0, 
\]
if $l-m$ is a negative or odd integer. Otherwise, by Theorem~\ref{thm:cd} and Theorem~\ref{thm:orthogonality}, 
\begin{align*}
\dfrac{1}{\omega_{\kappa, d}} \int_{S^{d-1}} q(y) p(y) h_{\kappa}^{2}(y) d\sigma(y)
& =\dfrac{1}{\omega_{\kappa, d}} \int_{S^{d-1}} q(y) p_m(y) h_{\kappa}^{2}(y) d\sigma(y) \\
& =\dfrac{1}{\omega_{\kappa, d}}\frac{1}{4^n n!\,(\lambda_{\kk}+1+m)_n} \int_{S^{d-1}}q(y) \Delta_{\kk}^n p(y)d\sigma(y) \\
& =\frac{1}{4^n n!\,(\lambda_{\kk}+1+m)_n} \frac{1}{2^m (\lambda_{\kk}+1)_m}q(\D)\Delta_{\kk}^n p \\
& = \dfrac{1}{2^{m+2n}n!\,(\lambda_{\kk}+1)_{m+n}} 
q(\mathscr{D}) \Delta_{\kappa}^{n} p.
\end{align*}

\end{proof}

If  $p\in\mathscr{H}_{\kk,l}$ in Theorem~\ref{thm:pizzettid1}, we recover Theorem~\ref{thm:orthogonality}. 
If  $q=1$  in Theorem~\ref{thm:pizzettid1}, we get Pizzetti's formula for homogeneous polynomials. 
Note that such a simple formula as in Theorem~\ref{thm:pizzettid1} does not holds for general $q\in\mathscr{P}_m$. 
For $\kk=0$, Theorem~\ref{thm:pizzettid1} was established by Bezubik, D\c{a}browska and Strasburger 
\cite[Corollary~2.1]{BDS}. Our proof closely 
follows their proof by using the canonical decomposition, the orthogonality relation, and the relation of 
two inner products on polynomials. For $\kk=0$, Theorem~\ref{thm:pizzettid1}  was also given by 
Estrada \cite[Proposition~3.3]{est0} with a proof similar to that of \cite[Corollary~2.1]{BDS}. 

An extension of Pizzetti's formula follows from 
 Theorem~\ref{thm:pizzettid1} and Taylor's theorem.

\begin{thm}\label{thm:expiz}
Suppose $q\in \mathscr{H}_{\kk,m}$. For a smooth function $f$ on 
a neighbourhood of \,$0\in\mathbb{R}^d$ 
we have 
\begin{equation}\label{eqn:exp1}
\dfrac{1}{\omega_{\kappa, d}} \int_{S^{d-1}} q(y) f(ry) h_{\kappa}^{2}(y) d\sigma(y)
 =\sum_{n=0}^N \frac{(q(\mathscr{D}) \Delta_{\kappa}^{n} f)(0)}{n!\,(\lambda_{\kk}+1)_{m+n}} 
 \left(\dfrac{r}{2}\right)^{m+2n}+o(r^{m+2N})
\end{equation}
as $r\to +0$. 

If $f$ is real analytic in the unit ball of radius 1 around $0\in\mathbb{R}^d$, then 
there exists a constant $\rho\in (0,1)$ such that 
\begin{equation}\label{eqn:exp2}
\dfrac{1}{\omega_{\kappa, d}} \int_{S^{d-1}} q(y) f(ry) h_{\kappa}^{2}(y) d\sigma(y)
 =\sum_{n=0}^\infty \frac{(q(\mathscr{D}) \Delta_{\kappa}^{n} f)(0)}{n!\,(\lambda_{\kk}+1)_{m+n}} 
 \left(\dfrac{r}{2}\right)^{m+2n}
\end{equation}
for any $r$ with $0<r<\rho$. 
\end{thm}
\begin{proof}
By Taylor's theorem $f(x)=\sum_{l=0}^{m+2N} p_l(x)+o(||x||^{m+2N})$ with $p_l\in\mathscr{P}_l$. 
By  Theorem~\ref{thm:pizzettid1}, we have
\begin{align*}
\dfrac{1}{\omega_{\kappa, d}} \int_{S^{d-1}} q(y) & f(ry) h_{\kappa}^{2}(y) d\sigma(y) \\
&=\dfrac{1}{\omega_{\kappa, d}} \sum_{n=0}^{N} r^{m+2n}\int_{S^{d-1}} q(y)p_{m+2n}(y)h_{\kappa}^{2}(y) d\sigma(y)+o(r^{m+2N}) \\
&=\sum_{n=0}^{N}  \frac{q(\mathscr{D})\Delta_{\kk}^n p_{m+2n}}{n!\,(\lambda_{\kk}+1)_{m+n}}\left(\frac{r}{2}\right)^{m+2n}+o(r^{m+2N}).
\end{align*}
If $f\in\mathscr{P}$, then it is clear that 
\begin{equation}
\label{eqn:ta}
q(\mathscr{D})\Delta_{\kk}^n p_{m+2n}=(q(\mathscr{D})\Delta_{\kk}^n f)(0).
\end{equation}
\eqref{eqn:ta} also holds for any smooth function $f$ by Taylor's formula associated with the Dunkl operators 
(cf. \cite[Corollary~2.17]{RV}). Hence \eqref{eqn:exp1} follows. 

Now assume that $f$ is real analytic in the unit ball of radius 1 around $0\in\mathbb{R}^d$. 
Then there exists $\rho\in (0,1)$ such that the multiple Taylor series 
$\sum_{l=0}^{\infty} p_l(x)\,\,\,(p_{l}\in\mathscr{P}_l)$ of $f$ converges normally to $f$ on 
 $\{x\in\mathbb{R}^d\,;\,||x||<\rho\}$. By  Theorem~\ref{thm:pizzettid1}, we have
\begin{align*}
\dfrac{1}{\omega_{\kappa, d}} \int_{S^{d-1}} q(y) & f(ry) h_{\kappa}^{2}(y) d\sigma(y) \\
&=\dfrac{1}{\omega_{\kappa, d}} \sum_{n=0}^\infty r^{m+2n}\int_{S^{d-1}} q(y)p_{m+2n}(y)h_{\kappa}^{2}(y) d\sigma(y) \\
&=\sum_{n=0}^\infty  \frac{q(\mathscr{D})\Delta_{\kk}^n p_{m+2n}}{n!\,(\lambda_{\kk}+1)_{m+n}}\left(\frac{r}{2}\right)^{m+2n}
\end{align*}
for any $r\in (0,\rho)$. 
By Taylor's formula associated with the  Dunkl operators 
(cf. \cite[Corollary~2.17]{RV}), 
 \eqref{eqn:ta} holds for a real analytic function $f$ around a neighbourhood of $0\in\mathbb{R}^d$, 
hence \eqref{eqn:exp2} follows. 
\end{proof}

For $\kk=0$, the above theorem for smooth $f$ is given by Estrada \cite[Theorem~5.1]{est2}. 

Putting $q= 1\in\mathscr{H}_{\kk,0}$ in the above theorem, we have an analogue of 
Pizzetti's formula as a corollary.

\begin{cor}\label{cor:piz}
For a smooth function $f$ on a neighbourhood of $0\in\mathbb{R}^d$ 
we have 
\begin{equation}\label{eqn:piz0}
\dfrac{1}{\omega_{\kappa, d}} \int_{S^{d-1}}  f(ry) h_{\kappa}^{2}(y) d\sigma(y)
 =\sum_{n=0}^N \frac{( \Delta_{\kappa}^{n} f)(0)}{n!\,(\lambda_{\kk}+1)_{n}} 
 \left(\dfrac{r}{2}\right)^{2n}+o(r^{2N})
\end{equation}
as $r\to +0$. 

If $f$ is real analytic in the unit ball of radius 1 around $0\in\mathbb{R}^d$, then 
there exists a constant $\rho\in (0,1)$ such that 
\begin{equation}\label{eqn:piz1}
\dfrac{1}{\omega_{\kappa, d}} \int_{S^{d-1}}  f(ry) h_{\kappa}^{2}(y) d\sigma(y)
 =\sum_{n=0}^\infty \frac{( \Delta_{\kappa}^{n} f)(0)}{n!\,(\lambda_{\kk}+1)_{n}} 
 \left(\dfrac{r}{2}\right)^{2n}, 
\end{equation}
\end{cor}

 Pizzetti's formula  associated with the Dunkl 
operators was  established by \cite{MT, ST}. In \cite{shimeno} we proved 
\eqref{eqn:piz1} for $f\in\mathscr{P}$ as an application of Hobson's formula  associated with the Dunkl 
operators. (Note there is an obvious mistake of unnecessary factor $(-1)^n$ in \cite[Corollary~4.5]{shimeno}.)  
Our proof given here provides an alternative proof of the formula. 
For $\kk=0$
\eqref{eqn:piz1}  gives original Pizzetti's formula (\cite{pizzetti, poritsky, AK, lysik}).

The right hand side of \eqref{eqn:exp2}, which we call the extended Pizzetti series, 
 is related with the Bessel function. 
For $\alpha\geq -1/2$, let $\varphi_{\alpha}$ denote the function defined by 
\[
\varphi_{\alpha}(x)=\varGamma(\alpha+1)\frac{J_\alpha(||x||)}{(||x||/2)^\alpha}
=\varGamma(\alpha+1)\sum_{n=0}^\infty \frac{(-1)^n}{n! \,\varGamma(\alpha+n+1)}\left(\frac{||x||}{2}\right)^{2n}, 
\]
where $J_\alpha$ denote the Bessel function of the first kind. 
The extended Pizzetti series  can be written as
\[
\sum_{n=0}^\infty  \frac{(q(\mathscr{D})\Delta_{\kk}^n f)(0)}{n!\,(\lambda_{\kk}+1)_{m+n}}\left(\frac{r}{2}\right)^{m+2n}
=\frac{1}{(\lambda_\kk)_m}\left(\frac{r}{2}\right)^m (q(\mathscr{D})\varphi_{\lambda_\kk+m}(\sqrt{-1}\mathscr{D}\,r)f)(0).
\]
Here $||\sqrt{-1}\mathscr{D}||^2$ in the series expansion of $\varphi_{\lambda_\kk+m}(\D)$  
is understood to be $-(\mathscr{D}_1^2+\cdots+\mathscr{D}_d^2)=-\Delta_\kk$. 
In particular, the right hand side of \eqref{eqn:piz1}, which we call the Pizzetti series, can be written as
\[
\sum_{n=0}^\infty  \frac{(\Delta_{\kk}^n f)(0)}{n!\,(\lambda_{\kk}+1)_{n}}\left(\frac{r}{2}\right)^{2n}
=(\varphi_{\lambda_\kk}(\sqrt{-1}\mathscr{D}\,r)f)(0).
\]

We have proved Pizzetti's formula (Corollary~\ref{cor:piz}) as a special case of our extended Pizzetti's formula 
(Theorem~\ref{thm:expiz}). 
Conversely, we will show that Theorem~\ref{thm:pizzettid1} and 
Theorem~\ref{thm:expiz} follows from Corollary~\ref{cor:piz} as an application of Hobson's formula. 
We recall  Hobson's formula associated with the Dunkl operators.

\begin{thm}[Hobson's formula \cite{shimeno}]\label{thm:hobson2}
Put $\rho=||x||$. 
If $p\in \mathscr{P}_m$, 
$f_0\in C^\infty((0,\infty))$, and $f(x)=f_0(\rho)$, then
\begin{equation}\label{eqn:hobson}
p(\D) f(x)=\sum_{i=0}^{[m/2]}\frac{1}{2^i i!}
\left[\left(\frac{1}{\rho}\frac{d}{d\rho}\right)^{m-i}\! f_0(\rho)\right] \Delta_\kk^i \,p(x).
\end{equation}
\end{thm}

\begin{prop}\label{prop:keyprop}
Put $\rho=||x||$. 
Let $m$ and $j$ be non-negative integers and $q\in \mathscr{H}_{m,\kk}$. 
If $j<m$ then $q(\D)r^{2j}=0$, while if $j\geq m$, then 
\[
q(\D)\rho^{2j}=\frac{2^m j!}{(j-m)!}\rho^{2j-2m}q(x).
\]
\end{prop}

\begin{proof}
Since $\Delta_\kk^i q=0$ for $i\geq 1$, 
it follows from Theorem~\ref{thm:hobson2} that
\begin{equation}\label{eqn:hcor}
q(\D)\rho^{2j}=\left[\left(\frac{1}{\rho}\frac{d}{d\rho}\right)^m \rho^{2j}\right] q(x),
\end{equation}
hence the results follow. 
\end{proof}

Assume $q\in \mathscr{H}_{\kk,m}$ and let $f(x)=\sum_{i=0}^{2J}p_i+o(||x||^{2J})\,\,(p_i\in\mathscr{P}_i)$ be the 
Taylor expansion. 
By the Pizzetti's formula \eqref{eqn:piz0}, 
\begin{equation}\label{eqn:piz4}
\frac{1}{\omega_{\kk,d}}\int_{S^{d-1}} q(ry)f(ry)h_\kk^2(y)d\sigma(y)=\sum_{j=0}^J \frac{ \Delta_{\kappa}^{j} (qf)(0)}{j!\,(\lambda_{\kk}+1)_{j}} 
 \left(\dfrac{r}{2}\right)^{2j}+o(r^{2J}).
\end{equation}
If $2j<m$, then $\Delta_{\kappa}^{j} (qf)(0)=0$, otherwise,  $\Delta_{\kappa}^{j} (qf)(0)=\Delta_{\kappa}^{j} (qp_{2j-m})$. 
If $2j\geq m$, then
\[
\Delta_{\kappa}^{j} (qp_{2j-m})  = \langle ||x||^{2j},q(x)p_{2j-m}(x)\rangle_\kk 
 =\langle q(\mathscr{D})||x||^{2j},p_{2j-m}(x)\rangle_\kk. 
\]
If $j\geq m$, then by Proposition~\ref{prop:keyprop},
\[
\Delta_{\kappa}^{j} (qp_{2j-m})  =  \frac{2^m j!}{(j-m)!}\langle ||x||^{2j-2m }q(x),p_{2j-m}(x)\rangle_\kk 
 =\frac{2^m j!}{(j-m)!}q(\mathscr{D})\Delta_\kk^{j-m}p_{2j-m}, 
\]
otherwise $\Delta_{\kappa}^{j} (qp_{2j-m})  =0$. 
Thus \eqref{eqn:piz4} become
\begin{align*}
\frac{r^m}{\omega_{\kk,d}}\int_{S^{d-1}} q(y) & f(ry)h_\kk^2(y)d\sigma(y) \\ & =
\sum_{j=m}^{J}\frac{2^m}{(j-m)!\,(\lambda_\kk+1)_j}(q(\mathcal{D})\Delta_\kk^{j-m}f)(0)\left(\frac{r}{2}\right)^{2j}+o(r^{2J}). 
\end{align*}
Hence we have \eqref{eqn:exp1} by putting $j-m=n,\,J=m+N$. 

\begin{rem}
We give a remark on the case of $\kk=0$. 
In this case, Proposition~\ref{prop:keyprop} was proved by Estrada \cite[Proposition~3.2]{est1} by using the canonical decomposition. 
Classical Hobson's formula \cite{Hobson1, N2, nomura} is a special case $\kk=0$ of Theorem~\ref{thm:hobson2}. 
Hobson's formula gives a simpler proof of Proposition~\ref{prop:keyprop} also for $\kk=0$. 
Estrada \cite[Proposition~6.1]{est2} deduced an extended Pizzetti's formula from Pizzetti's formula by 
using Proposition~\ref{prop:keyprop} for $\kk=0$. 

\end{rem}

\section{Applications}
We can prove 
an analogue of 
the Funk-Hecke formula as a corollary of Theorem~\ref{thm:pizzettid1}. 
Our proof follows closely to the proof of Funk-Hecke formula in superspace given by 
De Bie and Sommen~\cite[\S 7]{BS}. 

Let $V_\kk$ denote the Dunkl intertwining operator (cf. \cite[\S 6.5]{DX}). It is a linear operator that is 
uniquely determined by
\[
V_\kk \P_n\subset \P_n\,\,(n\in\mathbb{Z}_{\geq 0}),
\quad V_\kk1=1,\quad \D_\xi V_\kk=V_\kk\partial_\xi \,\,\,(\xi\in\mathbb{R}^d\setminus\{0\}).
\]

If $l-m=2n$ is an even non-negative integer and $q\in \mathscr{H}_m$, then by Theorem~\ref{thm:pizzettid1}, 
\begin{align}
\frac{1}{\omega_{\kk,d}}\int_{S^{d-1}}V_\kk & [\langle x,\cdot\,\rangle^l](y)q(y)h_\kk^2(y)d\sigma(y) \label{eqn:int000}\\
& = \frac{1}{2^{m+2n}n!\,(\lambda_\kk+1)_{m+n}}q(\mathscr{D})\Delta_\kk^n V_\kk[\langle x,\cdot\,\rangle^l] \notag \\
&= \frac{1}{2^{m+2n}n!\,(\lambda_\kk+1)_{m+n}}V_\kk [q(\partial)\Delta^n \langle x,\cdot\,\rangle^l ] \notag \\
& = \frac{(m+2n)!}{2^{m+2n}n!\,(\lambda_\kk+1)_{m+n}}q(x)  \notag 
\end{align}
for $x\in S^{d-1}$. Otherwise the integral in the left hand side of \eqref{eqn:int000} is zero. 

On the other hand, for a continuous function $\varphi$ on $[-1,1]$, define
\begin{equation}\label{eqn:int002}
a_{\kk,m}(\varphi)=\frac{m!}{(2\lambda_\kk)_m B\left(\lambda_\kk+\frac12,\frac12\right)}
\int_{-1}^1 \varphi(t) \,C_m^{\lambda_\kk}(t)(1-t^2)^{\lambda_\kk-\frac12}dt. 
\end{equation}
Here $C_m^\lambda(t)$ is the Gegenbauer polynomial
\[
C_m^\lambda(t)=\frac{(-1)^m (2\lambda)_m}{2^m m! \,(\lambda+\frac12)_m}
(1-t^2)^{\frac12-\lambda}\frac{d^m}{dt^m}(1-t^2)^{\lambda+m-\frac12}
\]
and $B(\,\cdot,\,\cdot\,)$ is the Beta function. 
Then, by integration by parts, 
\[
a_{\kk,m}(t^l)=\frac{\varGamma(\lambda_\kk+1)}{2^m \sqrt{\pi}\varGamma  (\lambda_\kk+m+\frac12)}\int_{-1}^1
\left(\frac{d^m}{dt^m}t^l\right) (1-t^2)^{\lambda_\kk+m-\frac12}dt.
\]
Here we used basic properties of the Beta function, the Gamma function $\varGamma(\,\cdot\,)$, 
and the shifted factorial. 
If $l-m=2n$ is an even non-negative integer, then 
\begin{align}
a_{\kk,m}(t^l)& =\frac{\varGamma(\lambda_\kk+1)}{2^m \sqrt{\pi}\varGamma  (\lambda_\kk+m+\frac12)
}\frac{l!}{(l-m)!}
B\left(n+\frac12,\lambda_\kk+m+\frac12\right) \label{eqn:int003} \\
& =\frac{(m+2n)!}{2^{m+2n}n!\,(\lambda_\kk+1)_{m+n}}
, \notag
\end{align}
otherwise $a_{\kk,m}(t^l)=0$. By \eqref{eqn:int000} and \eqref{eqn:int003}, 
\begin{equation}\label{eqn:int004}
\frac{1}{\omega_{\kk,d}}\int_{S^{d-1}}V_\kk  [\langle x,\cdot\,\rangle^l](y)q(y)h_\kk^2(y)d\sigma(y)
=a_{\kk,m}(t^l)\,q(x)
\end{equation}
for $l,\,m\in \mathbb{Z}_{\geq 0}$, $q\in \mathscr{H}_{\kk,m}$, and $x\in S^{d-1}$. 

By \eqref{eqn:int004} and the Weierstrass approximation theorem, we have the 
following result. 

\begin{cor}[Funk-Hecke formula \cite{xu}, {\cite[Theorem~7.3.4]{DX}}]\label{cor:fh}
Let $\varphi$ be a continuous function on $[-1,1]$ and $q\in\mathscr{H}_{\kk,m}$. Then
\[
\frac{1}{\omega_{\kk,d}}\int_{S^{d-1}}V_\kk [ \varphi(\langle x,\cdot\,\rangle)](y)q(y)h_\kk^2(y)d\sigma(y)
=a_{\kk,m}(\varphi)\,q(x)\quad (x\in S^{d-1}).
\]
\end{cor}

Original proof of the Funk-Hecke formula \cite{xu} (also in the classical case \cite{DaX0}) 
uses the reproducing kernel. 
We did not use the reproducing kernel in our proof. 
By the Funk-Hecke formula, 
we can prove the following reproducing property. 

\begin{cor}[{\cite[Theorem~3.2]{xu97}}]
Define
\[
P_{\kk,n}(x,y)=\frac{n+\lambda_\kk}{\lambda_\kk}V_\kk\!\left[C_n^{\lambda_\kk}(\langle x,\cdot\rangle\right]\!(y).
\]
Then for any $q\in\mathscr{H}_{\kk,m}$, 
\[
\frac{1}{\omega_{\kk,d}}\int_{S^{d-1}}P_{\kk,n}(x,y)q(y)h_\kk^2(y)d\sigma(y)
=\delta_{mn}\,q(x)\quad ( x\in S^{d-1}).
\]
\end{cor}
\begin{proof}
Assume $q\in\mathscr{H}_{\kk,m}$. 
By applying Corollary~\ref{cor:fh} for $\varphi(t)=C_n^{\lambda_\kk}(t)$,
\[
\frac{1}{\omega_{\kk,d}}\int_{S^{d-1}}V_\kk\!\left[C_n^{\lambda_\kk}(\langle x,\cdot\rangle\right]\!
(y)q(y)h_\kk^2(y)d\sigma(y)
=a_{\kk,m}(C_n^{\lambda_\kk})\,q(x)\quad (x\in S^{d-1}).
\]
Here
\begin{align*}
a_{\kk,m}(C_n^{\lambda_\kk})& =
\frac{m!}{(2\lambda_\kk)_m B\left(\lambda_\kk+\frac12,\frac12\right)}
\int_{-1}^1  C_n^{\lambda_\kk}(t)\, C_m^{\lambda_\kk}(t)(1-t^2)^{\lambda_\kk-\frac12}dt \\
& =\delta_{mn}\frac{\lambda_\kk}{n+\lambda_\kk}
\end{align*}
by the orthogonality relations for the Gegenbauer polynomials (cf. \cite[\S 1.4.3]{DX}). 
\end{proof}

\bigskip

\end{document}